\documentclass[12pt]{amsart}

\usepackage{hyperref}  

\usepackage{graphicx}
\usepackage{amsfonts}
\usepackage{amssymb}
\usepackage{amsmath}
\usepackage{amsthm}

\newtheorem{theorem}{Theorem}[section]

\newtheorem{lemma}[theorem]{Lemma}
\newtheorem{claim}{}[theorem]
\newtheorem{conjecture}[theorem]{Conjecture}

\newcommand{\del}{\backslash}
\newcommand{\con}{/}
\newcommand{\cB}{\mathcal{B}}
\newcommand{\bZ}{\mathbb Z}
\newcommand{\bR}{\mathbb R}

\begin{document}

\sloppy

\title[Frame matroids and lifted-graphic matroids]{Infinitely many excluded minors for frame matroids and for lifted-graphic matroids}

\author{Rong Chen}
\address{Center for Discrete Mathematics, Fuzhou University,
Fuzhou, P. R. China}
\email{rongchen@fzu.edu.cn}
\thanks{This research was supported by grants from  
the Office of Naval Research [N00014-10-1-0851] and NSERC [203110-2011].}

\author{Jim Geelen}
\address{Department of Combinatorics and Optimization,
University of Waterloo, Waterloo, Canada}
\email{jim.geelen@uwaterloo.ca}

\date{\today}

\begin{abstract}
We present infinite sequences of excluded minors for both the class of lifted-graphic matroids and the class of frame matroids.
\end{abstract}

\maketitle

\section{Introduction}
A matroid $M$ is a {\em frame matroid} if there is a
matroid $M'$ with a basis $V$ such that
$M=M'\del V$ and, for each $e\in E(M)$, the unique
circuit in $V\cup \{e\}$ has size at most $3$.
A matroid $M$ is {\em lifted-graphic} if there is a matroid
$M'$ with $E(M')=E(M)\cup\{e\}$ such that
$M'\del e= M$ and $M'\con e$ is graphic.
The classes of lifted-graphic matroids and
frame matroids were introduced by Zaslazsky~\cite{Zaslavsky-II}
who proved that they are minor-closed.

We dispel the widespread belief that these classes would
likely have only finitely many excluded minors.
\begin{theorem}\label{main-frame}
There exist infinitely many pairwise non-isomorphic
excluded minors for the class of frame matroids.
\end{theorem}

\begin{theorem}\label{main-lifted}
There exist infinitely many pairwise non-isomorphic
excluded minors for the class of lifted-graphic matroids.
\end{theorem}

Our excluded-minors are based on constructions introduced by Chen and Whittle~\cite{CW17}. 
DeVos, Funk, and Pivotto \cite{DFI, Funk} characterised the non-$3$-connected excluded-minors for the class of frame matroids.

The existence of an infinite set of excluded minors
does not necessarily prevent us from describing a class
explicitly; see, for example, Bonin's excluded minor
characterization for the class of lattice-path matroids~\cite{Bonin}. 
We believe that the excluded-minors for both the class of frame matroids and the class of lifted-graphic matroids
are highly structured, and that it may be possible to obtain an explicit characterization of the 
sufficiently large excluded minors. Towards this end we pose the following conjectures.

\begin{conjecture}\label{conj-lift}
There exists an integer $k$ such that
all excluded minors for the class of lifted-graphic
matroids  have branch-width at most $k$.
\end{conjecture}

\begin{conjecture}\label{conj-frame}
There exists an integer $k$ such that
all excluded minors for the class of
frame matroids have branch-width at most $k$.
\end{conjecture}

The class of quasi-graphic matroids,
introduced in~\cite{GGW}, contains
both the lifted-graphic matroids and the frame
matroids. The infinitely many excluded minors given in this paper for the class of frame matroids and the class of lifted-graphic matroids  are quasi-graphic. 
In contrast to Theorems~\ref{main-frame}
and~\ref{main-lifted} we remain confident that
the class of quasi-graphic matroids admits a
finite excluded-minor characterization.
\begin{conjecture}\label{conj-quasi}
There are, up to isomorphism, only finitely
many excluded-minors for the class of
quasi-graphic matroids.
\end{conjecture}

In support of this conjecture, Chen \cite{Chen17} recently proved that there are only two $8$-connected excluded minors for the class of quasi-graphic matroids;
namely $U_{3,7}$ and $U_{4,7}$. 

\section{Preliminaries}
We assume that the reader is familiar with matroid theory
and we follow the terminology of Oxley~\cite{Oxley}.

Recall that a {\em circuit-hyperplane} of a matroid $M$
is a set $C$ that is both a circuit and a hyperplane, and
that we can obtain a new matroid $M'$ by relaxing a
circuit-hyperplane $C$ of $M$;
see \cite[Proposition~1.5.14]{Oxley}.
More specifically, $\cB(M') = \cB(M)\cup \{C\}$ 
where $\mathcal B(M)$ is the set of bases of $M$.

The reverse operation was introduced by Chen and Whittle~\cite{CW17}.
A {\em free basis} of a matroid $M$ is a basis $B$ such that 
$B\cup\{e\}$ is a circuit for each $e\in E(M)-B$.
If $B$ is a free basis of $M$ then $(E(M),\cB(M)-\{B\})$ 
is a matroid (see~\cite{CW17}); we say that
$(E(M),\cB(M)-\{B\})$ is obtained by {\em tightening} $B$.

Let $G$ be a graph. For $v\in V(G)$ we let $\delta_G(v)$
denote the set of edges incident with $v$.
For any $U\subseteq V(G)$ and $F\subseteq E(G)$, let $G[U]$
be the induced subgraph of $G$ defined on $U$, and let $G[F]$
be the subgraph of $G$ with $F$ as its edge set and without
isolated vertices. A {\em cycle} of a graph is a
connected $2$-regular subgraph.

We assume that the reader is familiar with  bias graphs;
see Zaslavsky~\cite{Zaslavsky-I}.
Let $(G,\cB)$ be a bias graph. The cycles in $\cB$
are called {\em balanced} and a subgraph $H$ of
$G$ is {\em balanced} if each of the cycles in $H$ is
balanced. A set $F\subseteq E(G)$ is {\em balancing}
if $G\del F$ is balanced.

For this paper it is more convenient to use
the bias graph definition of frame matroids, which is in fact the way that they were originally
defined by Zaslavsky~\cite{Zaslavsky-II}.
Let $(G,\cB)$ be a bias graph. 
We define $FM(G,\cB)$ to be the matroid
with ground set $E(G)$ such that $I\subseteq E(G)$
is independent if and only if $G[I]$ has no balanced 
cycles and for each component $H$ of $G[I]$ we have
$|E(H)|\le |V(H)|$. Henceforth we will call a matroid $M$ a 
{\em frame matroid} if and only if $M=FM(G,\cB)$ for 
some bias graph $(G,\cB)$; 
Zaslavsky~\cite{Zaslavsky-Frame} proved that this
definition is equivalent to the geometric definition
stated in the introduction. 

Let $M$ be a matroid. If $(G,\cB)$ is a biased graph
such that $M=FM(G,\cB)$, then $\cB$ is implicitly
determined by $G$ (and $M$). Hence we refer to the graph
$G$, itself, as a {\sl frame representation} of $M$,
and given a frame representation of a matroid we will
refer to its cycles as balanced or non-balanced accordingly.

As with frame matroids, it is also more convenient to use
the bias graph definition of lifted-graphic matroids;
see Zaslavsky~\cite{Zaslavsky-II}.
We define $LM(G,\cB)$ to be the matroid 
with ground set $E(G)$ such that $I\subseteq E(G)$
is independent if and only if $G[I]$ has at most one cycle
and, should it exist, that cycle is non-balanced.
Henceforth we will call a matroid $M$ a 
{\em lifted-graphic matroid} if and only if $M=LM(G,\cB)$ for 
some bias graph $(G,\cB)$; 
Zaslavsky~\cite{Zaslavsky-Lifted} showed that this new
definition is equivalent to the earlier definition
stated in the introduction.

Let $M$ be a matroid. If $(G,\cB)$ is a biased graph
such that $M=LM(G,\cB)$, then $\cB$ is implicitly
determined by $G$ (and $M$). Hence we refer to the graph
$G$, itself, as a {\sl lifted-graphic representation} of $M$,
and given a lifted-graphic representation of a matroid we will
refer to its cycles as balanced or non-balanced accordingly. 

One well-known way to construct a bias graph is via a group-labelled
graph (also known as a {\em gain graph} and a {\em voltage graph}). 
Here we use only the group of integers under addition,
which we denote by $\bZ$, and the group of non-zero real numbers
under multiplication, which we denote by $\bR^{\times}$.
For an abelian group $\Gamma$, a {\em $\Gamma$-labelled graph} is a pair $(\vec G,\gamma)$ where
$\vec G$ is an oriented graph and $\gamma:E(\vec G)\rightarrow \Gamma$.  
Let $(\vec G,\gamma)$ be a $\Gamma$-labelled graph and let 
$G$ be the underlying graph of $\vec G$.
A cycle $C$ of $G$ is {\em balanced} if 
the group-product of the labels on ``clockwise" oriented edges is 
equal to the group-product of the labels on
``counter-clockwise" oriented edges; this is independent
of the direction on $C$ we choose as clockwise.
If $\cB$ is the set of balanced cycles of $G$
then $(G,\cB)$ is a biased graph.

The following construction, due to Zaslavsky \cite{Zaslavsky03}, builds an
$\bR$-representable frame matroid from an $\bR^+$-labelled graph $(\vec G,\gamma)$.
We will assume that $(\vec G,\gamma)$ has no loops.
Let $A$ be a $V(\vec G)\times E(\vec G)$ matrix over $\bR$ where,
for a vertex $v$ and edge $e$, we have
$A_{v,e} = 1$ if $v$ is the tail of $v$, $A_{v,e} = -\gamma(e)$ if $v$ is the head of $e$,
and $A_{v,e}=0$ otherwise. Then 
$M(A) = FM(G,\cB)$ where  $(G,\cB)$ is the bias graph associated with $(\vec G,\gamma)$.

Zaslavsky \cite{Zaslavsky03} also showed how to build an
$\bR$-representable lifted-graphic matroid from a $\bZ$-labelled graph $(\vec G,\gamma)$.
Again will assume that $(\vec G,\gamma)$ has no loops. Let $B$ be the signed incidence
matrix of $\vec G$; thus $B\in\{0,\pm 1\}^{V(\vec G)\times E(\vec G)}$ where
$B$ is $1$ or $-1$ when $v$ is the head or tail, respectively, of $e$. Now construct
a matrix $A$ by appending the vector $\gamma\in \bZ^{E(\vec G)}$ as a new row to $B$.
Then 
$M(A) = LM(G,\cB)$ where $(G,\cB)$ is the bias graph associated with $(\vec G,\gamma)$.

A cocircuit $C^*$ of a matroid $M$ is {\em non-separating}
if $M\del C^*$ is connected.
If $C^*$ is a non-separating cocircuit of a matroid $M$
and $M=FM(G,\cB)$, then either 
$C^*$ is a balancing set of $(G,\cB)$ or
$C^* = \delta_G(v)$ for some vertex $v\in V(G)$.

\section{Frame matroids}
In this section we prove Theorem~\ref{main-frame}. Let $k\geq 7$
be an odd integer. (The condition
that $k\ge 7$ is to simplify the proof; $k\ge 3$ suffices.)
Let $(\vec G_k,\gamma)$ be the $\bR^{\times}$-labelled graph
defined in Figure~\ref{G_n} and
let $G_k$ denote its underlying undirected graph.
Let $\cB$ denote the balanced cycles of  $(\vec G_k, \gamma)$
and let $N_k = FM(G_k,\cB)$. 

\begin{figure}[htbp]
\begin{center}
\includegraphics[page=1,height=8cm]{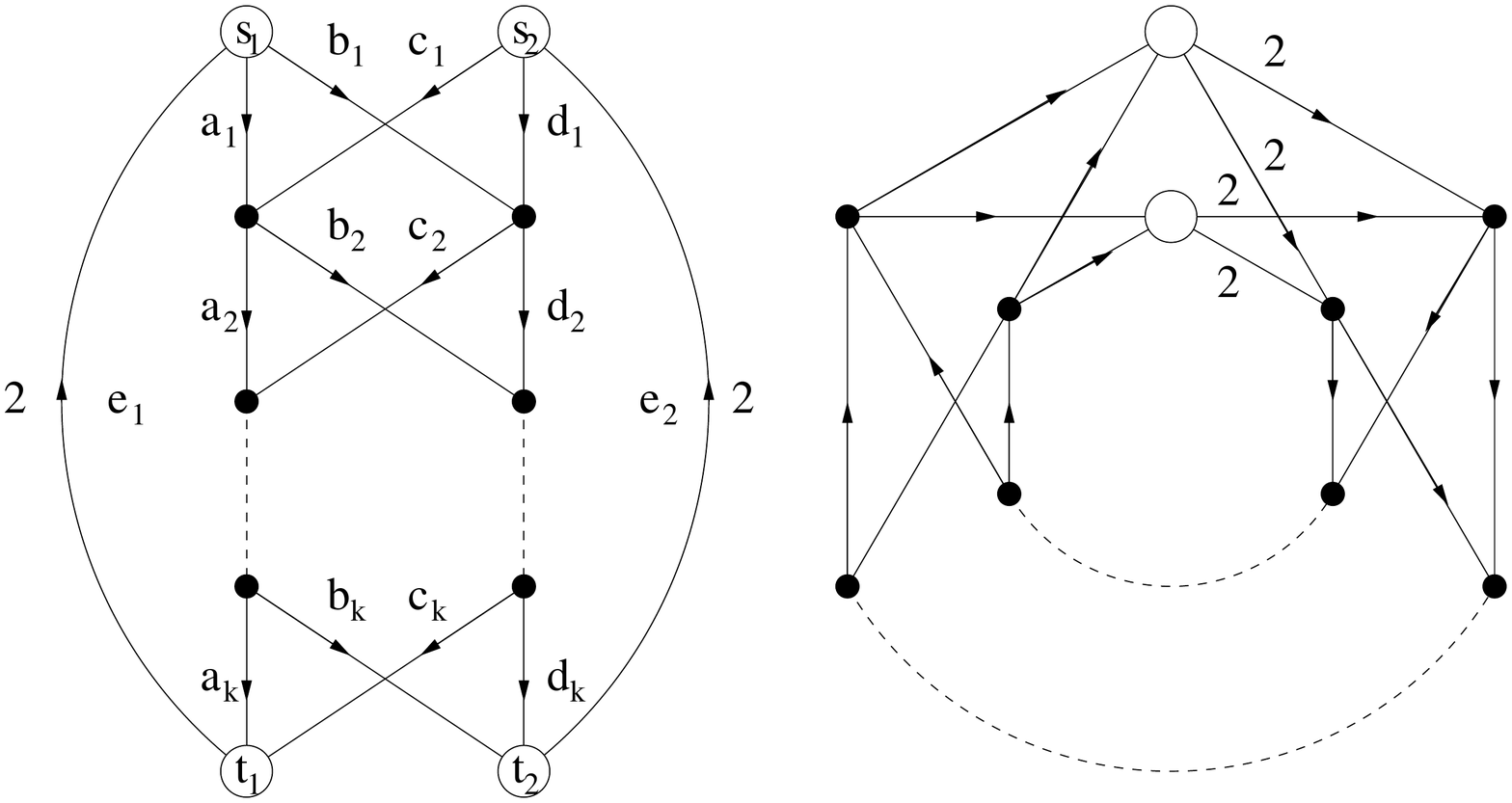}
\caption{The graphs $G_k$ and $G_k\con\{e_1,e_2\}$, respectively, with
their $\bR^{\times}$-labellings; unlabelled edges have group-label
$1$, where $a_i, b_i,c_i,d_i$ are not group labels but edge names.}
\label{G_n}
\end{center}
\end{figure}

Let
$ P = \{a_1,\ldots,a_{k}\}\cup \{d_1,\ldots,d_{k}\}$  and
$ Q = \{b_1,\ldots,b_{k}\}\cup \{c_1,\ldots,c_{k}\}.$
Note that $P\cup\{e_1,e_2\}$ and $Q\cup\{e_1,e_2\}$ are free bases of 
$N_k$; let $M_k^F$ be the matroid obtained
from $N_k$ by tightening $P\cup\{e_1,e_2\}$ and
$Q\cup\{e_1,e_2\}$. Thus $P\cup\{e_1,e_2\}$ and $Q\cup\{e_1,e_2\}$
are circuits of $M_k^F$. 

We will prove that $M_k^F\con\{e_1,e_2\}$ is an excluded minor.
We start with the easier task of showing that 
proper minors of $M_k^F\con\{e_1,e_2\}$ are frame matroids.
\begin{lemma}\label{easy lemma}
For each $e\in P\cup Q$, both $M_k^F/e$ and $M_k^F\del e$ are frame
matroids. 
\end{lemma}

\begin{proof}
Let $M_P$ and $M_Q$ denote the matroids obtained from 
$M_k^F$ by relaxing the circuit hyperplanes 
$P\cup\{e_1,e_2\}$ and $Q\cup \{e_1,e_2\}$ respectively.
For each $e\in P$, we have
$M_k^F\del e = M_P\del e$ and $M_k^F\con e = M_Q\con e$.
Similarly, for each $e\in Q$, we have
$M_k^F\con e = M_P\con e$ and $M_k^F\del e = M_Q\del e$.
So it suffices to prove that $M_P$ and $M_Q$ are frame matroids.

Note that $M_P$ and $M_Q$ are obtained from 
$N_k$ by tightening the free bases
$Q\cup\{e_1,e_2\}$ and $P\cup \{e_1,e_2\}$ respectively.
Since $G_k$ is a frame representation of $N_k$ and
$G_k[Q\cup\{e_1,e_2\}]$ is a cycle in $G_k$, 
we have that 
$G_k$ is a frame representation of $M_P$; so 
$M_P$ is indeed a frame matroid.
Let $G'_k$ be the graph obtained from $G_n\del \{e_1,e_2\}$
by adding $e_1$ connecting $s_1$ to $t_2$ and adding $e_2$
connecting $s_2$ to $t_1$. It is straightforward to
verify that $G'_k$ is a frame representation of $N_k$
(since $\{e_1,e_2\}$ is a series pair in $N_k$).
Finally, since $G'_k[P\cup\{e_1,e_2\}]$ is a cycle in $G'_k$, 
we have that $G'_k$ is a frame representation of $M_Q$; so 
$M_Q$ is indeed a frame matroid.
\end{proof}

Now it remains to show that $M_k^F/\{e_1,e_2\}$ itself
is not a frame matroid.
\begin{lemma}\label{key lemma}
$M_k^F/\{e_1,e_2\}$ is not a frame matroid.
\end{lemma}
\begin{proof}
Assume to the contrary that $H$ is a frame representation of
$M_k^F/\{e_1,e_2\}$.
Let $C_i=\{a_i,b_i,c_i,d_i\}$ for each $i\in\{1,\ldots,k\}$ and
let $G=G_k\con \{e_1,e_2\}$.
Figure~\ref{G_n} depicts $G$ with a group labelling
encoding the balanced cycles with
respect to $N_k\con \{e_1,e_2\}$. From this group labelled 
graph we see that:
\begin{itemize}
\item[(i)] each cocircuit in $N_k\con\{e_1,e_2\}$ (and hence also 
in $M_k^F\con\{e_1,e_2\}$) has size at least $4$,
\item[(ii)] for each $4$-element cycle $C$ of $G$, the set $E(C)$ 
is a circuit 
in $N_k\con\{e_1,e_2\}$ and, hence, also in
$M_k^F\con \{e_1,e_2\}$, 
\item[(iii)] for each $i\in\{1,\ldots,k\}$, the set $C_i$ is 
a non-separating
cocircuit in $N_k\con\{e_1,e_2\}$ and, hence, also in
$M_k^F\con \{e_1,e_2\}$, and
\item[(iv)] for each $v\in V(G)$, the set $\delta_G(v)$
is a $4$-element non-separating cocircuit in
$N_k\con\{e_1,e_2\}$ and, hence, also in
$M_k^F\con \{e_1,e_2\}$.
\end{itemize}

\begin{claim}\label{4-regular}
$H$ is a simple connected $4$-regular graph.
\end{claim}

\begin{proof}[Subproof.]
Since $N_k\con\{e_1,e_2\}$ is connected, so is  $M^F_k\con\{e_1,e_2\}$. Then $H$ is connected.
Since $|E(H)| = |P\cup Q| = 4k = 2|V(H)|$,
we see that $H$ has average degree $4$. It follows from 
$(i)$ that $H$  is $4$-regular. It remains to show that
$H$ is simple; suppose otherwise and let $C$ be a cycle
of length at most $2$.
At least $3k-6$ of the non-separating cocircuits described in
(iii) and (iv) are disjoint from $E(C)$.
Since $k\ge 7$ we have $3k-6>|V(H)|$ and hence one
of these non-separating cocircuits is balancing.
But then $C$ is a circuit of $M_k^F\con\{e_1,e_2\}$, a contradiction to the fact that $M_k^F\con\{e_1,e_2\}$ is a simple matroid. 
\end{proof}

For each 4-element circuit $C$ of $M_k^F\con\{e_1,e_2\}$, since $H$ is simple, $C$ is a cycle of $H$. In particular, 
for each $4$-cycle $C$ of $G$, since $E(C)$ is a circuit of $M_k^F\con\{e_1,e_2\}$,  the set $E(C)$ is a cycle of $H$. 
Since $C_1$, $\{a_1,b_1,a_2,c_2\}$, and $\{a_1,c_1,a_k,b_k\}$ are cycles of the simple 4-regular graph $H$, the sets
$\{a_1,d_1\}$ and $\{b_1,c_1\}$ are matchings in $H$. Repeating the analysis $k$-times, 
it is routine to show that $H$ is isomorphic to $G$ and,
moreover, that 
\begin{itemize}
\item[(a)] there is an isomorphism that fixes
$C_1,\ldots, C_k$ set-wise, and
\item[(b)] for each $i\in\{1,\ldots,k\}$, the sets
$\{a_i,d_i\}$ and $\{b_i,c_i\}$ are matchings in $H$.
\end{itemize}
Now, since $k$ is odd, one of 
$H[P]$ and $H[Q]$ is a cycle while the other
is the union of two vertex-disjoint cycles.
However $P$ and $Q$ are both circuits in $M_k^F\con\{e_1,e_2\}$
which contradicts the fact that $H$ is a frame representation.
\end{proof}


\section{Lifted-graphic matroids}
In this section we prove Theorem \ref{main-lifted}. Let $k\geq 3$
be an odd integer.  Let $(\vec G_k,\gamma)$ be the
$\bZ$-labelled graph defined in Figure~\ref{G_n+1} and let
$G_k$ denote its underlying undirected graph.
Let $\cB$ denote the balanced cycles of  $(\vec G_k, \gamma)$
and let $N_k = LM(G_k,\cB)$. 

\begin{figure}[htbp]
\begin{center}
\includegraphics[page=2,height=8cm]{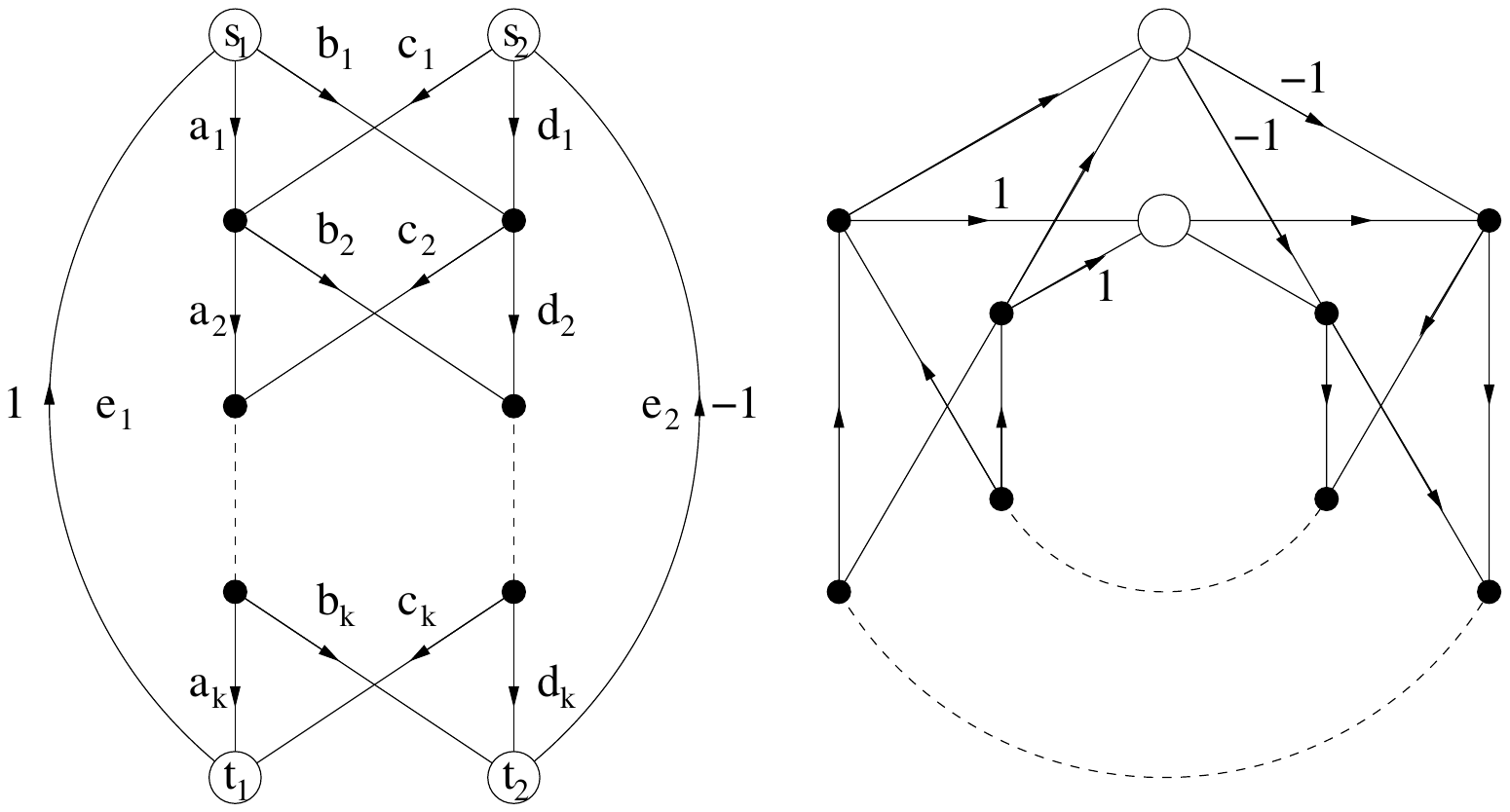}
\caption{The graphs $G_k$ and $G_k\con \{e_1,e_2\}$, respectively, with their $\bZ$-labellings;
unlabelled edges have group-value $0$, where $a_i, b_i,c_i,d_i$ are not group labels but edge names.}
\label{G_n+1}
\end{center}
\end{figure}

Let
$ P = \{a_1,\ldots,a_{k}\}\cup \{d_1,\ldots,d_{k}\}$  and
$ Q = \{b_1,\ldots,b_{k}\}\cup \{c_1,\ldots,c_{k}\}.$
Note that $P\cup\{e_1,e_2\}$ and $Q\cup\{e_1,e_2\}$ are 
circuit-hyperplanes of 
$N_k$; let $M_k^L$ be the matroid obtained
from $N_k$ by relaxing $P\cup\{e_1,e_2\}$ and
$Q\cup\{e_1,e_2\}$. 

We will prove that $M_k^L\con\{e_1,e_2\}$ is an excluded minor.
We start by showing that
proper minors of $M_k^L\con\{e_1,e_2\}$ are lifted-graphic matroids;
this is almost a carbon copy of the proof
of Lemma~\ref{easy lemma}.
\begin{lemma}\label{easy lemma+1}
For each $e\in P\cup Q$, both $M_k^L/e$ and $M_k^L\del e$ are 
lifted-graphic matroids. 
\end{lemma}

\begin{proof}
Let $M_P$ and $M_Q$ denote the matroids obtained from 
$M_k^L$ by tightening the free bases
$P\cup\{e_1,e_2\}$ and $Q\cup \{e_1,e_2\}$ respectively.
For each $e\in P$, we have
$M_k^L\del e = M_P\del e$ and $M_k^L\con e = M_Q\con e$.
Similarly, for each $e\in Q$, we have
$M_k^L\con e = M_P\con e$ and $M_k^L\del e = M_Q\del e$.
So it suffices to prove that $M_P$ and $M_Q$ are lifted-graphic matroids.

Note that $M_P$ and $M_Q$ are obtained from 
$N_k$ by relaxing the circuit-hyperplanes
$Q\cup\{e_1,e_2\}$ and $P\cup \{e_1,e_2\}$ respectively.
Since $G_k$ is a lifted-graphic representation of $N_k$ and
$G_k[Q\cup\{e_1,e_2\}]$ is a cycle in $G_k$, 
we have that 
$G_k$ is a lifted-graphic representation of $M_P$; so 
$M_P$ is indeed a lifted-graphic matroid.
Let $G'_k$ be the graph obtained from $G_k\del \{e_1,e_2\}$
by adding $e_1$ connecting $s_1$ to $t_2$ and adding $e_2$
connecting $s_2$ to $t_1$. It is straightforward to
verify that $G'_k$ is a lifted-graphic representation of $N_k$
(since $\{e_1,e_2\}$ is a series pair in $N_k$).
Finally, since $G'_k[P\cup\{e_1,e_2\}]$ is a cycle in $G'_k$, 
we have that $G'_k$ is a lifted-graphic representation of $M_Q$; so 
$M_Q$ is indeed a lifted-graphic matroid.
\end{proof}

Now it remains to show that $M_k^L/\{e_1,e_2\}$ itself
is not a lifted-graphic matroid.
\begin{lemma}\label{key lemma+1}
$M_k^L/\{e_1,e_2\}$ is not a lifted-graphic matroid.
\end{lemma}
\begin{proof}
Assume to the contrary that $H$ is a lifted-graphic representation of
$M_k^L/\{e_1,e_2\}$. Since $N_k\con\{e_1,e_2\}$ is connected, $M_k^L\con\{e_1,e_2\}$ is connected. 
Since identifying two vertices in different components of a bias graph does not change its lifted-graphic matroid, we may assume that $H$ is connected. 
Let $A_1 = \{a_1,b_1,a_k,c_k\}$, $A_2 = \{c_1,d_1,b_k,d_k\}$,
$B_1 = \{a_1,b_1,b_k,d_k\}$, $B_2 = \{c_1,d_1,a_k,c_k\}$,
$C_1 = \{a_1,b_1,c_1,d_1\}$, and $C_2 = \{a_k,b_k,c_k,d_k\}$.
Let $G=G_k\con \{e_1,e_2\}$;
Figure~\ref{G_n+1} depicts $G$ with a $\bZ$-labelling
encoding its balanced cycles with
respect to $N_k\con \{e_1,e_2\}$. From this $\bZ$-labelled 
graph we see that:
\begin{itemize}
\item[(i)] each cocircuit in $N_k\con\{e_1,e_2\}$ (and hence also 
in $M_k^L\con\{e_1,e_2\}$) has size at least $4$, and
\item[(ii)] The only $4$-element cocircuits of $N_k\con\{e_1,e_2\}$
(and hence also of $M_k^L\con\{e_1,e_2\}$) are
the sets $\delta_{G_k}(v)$ for $v\in V(G_k)-\{s_1,s_2,t_1,t_2\}$
and the sets $A_1,$ $A_2,$ $B_1,$ $B_2,$ $C_1,$ and $C_2$. 
\end{itemize}

\begin{claim}\label{4-regular+1}
$H$ is a loopless connected $4$-regular graph.
\end{claim}

\begin{proof}[Subproof.]
Since $|E(H)| = |P\cup Q| = 4k = 2|V(H)|$,
we see that $H$ has average degree $4$. It follows from 
$(i)$ that $H$ is $4$-regular and loopless. 
\end{proof}

We will call a set $X\subseteq E(H)$ {\em vertical}
if there exists $v\in V(H)$ such that $X=\delta_H(v)$.
Now each of the $2k$ vertical sets is a $4$-element
cocircuit of $M_k^L\con\{e_1,e_2\}$ and each element
in $P\cup Q$ is in exactly two vertical sets.
We have listed all of the $4$-element cocircuits
of $M_k^L\con\{e_1,e_2\}$ in $(ii)$.
Note that the elements in
$\{a_2,a_3,\ldots,a_{k-1}\}\cup\{d_2,d_3,\ldots,d_{k-1}\}$ 
are each in exactly two $4$-element cocircuits.
It follows that, for each $v\in V(G_n)-\{s_1,s_2,t_1,t_2\}$,
the set $\delta_{G_k}(v)$ is vertical.
There are three possibilities for the pair of remaining 
vertical sets, namely, $(A_1,A_2)$, $(B_1,B_2)$, and $(C_1,C_2)$.

First suppose that $C_1$ and $C_2$ are both vertical.
Then $H[\{a_1,c_1,a_k,b_k\}]$ is the union 
of two edge-disjoint cycles. So
$\{a_1,c_1,a_k,b_k\}$ is dependent in
$M_k^L\con \{e_1,e_2\}$ and hence also in
$N_k\con \{e_1,e_2\}$. However, from the definition
of $N_k$, the set $\{e_1,e_2,a_1,c_1,a_k,b_k\}$
is independent. 
From this contradiction we have that the remaining 
pair of vertical sets is either
$(A_1,A_2)$ or $(B_1,B_2)$.

Now, since $k$ is odd, one of 
$H[P]$ and $H[Q]$ is a cycle while the other
is the union of two vertex-disjoint cycles.
However $P$ and $Q$ are both independent in $M_k^L\con\{e_1,e_2\}$
which contradicts the fact that $H$ is a lifted-graphic representation.
\end{proof}

\section*{Acknowledgements} We thank the referees for their careful reading of this
manuscript and their detailed comments.

\end{document}